\documentclass{article}

\usepackage{array}
\usepackage{amsmath}
\usepackage{amssymb}
\usepackage{empheq}
\usepackage{mathtools}
\usepackage{multicol}
\usepackage{multirow}
\usepackage[linesnumbered]{algorithm2e}
\usepackage{amsthm}
\usepackage[a4paper, total={16cm, 26cm}]{geometry}

\newtheorem{assumption}{Assumption}
\newtheorem{proposition}{Proposition}

\newcommand{\s}[1]{\{#1\}}
\newcommand{\N}{\mathcal N}
\newcommand{\M}{\mathcal M}
\newcommand{\K}{\mathcal K}
\newcommand{\fold}{(P_1)}

\newcommand{\fa}{$(P_2)$}

\newcommand{\fb}{$(CP_1)$}

\newcommand{\fc}{$(CP_2)$}

\newcolumntype{R}[1]{>{\raggedleft\arraybackslash}m{#1}}
\newcolumntype{L}[1]{>{\raggedright\arraybackslash}m{#1}}
\newcolumntype{M}[1]{>{\centering\arraybackslash}m{#1}}
\newcolumntype{H}{>{\setbox0=\hbox\bgroup}c<{\egroup}@{}}
\SetKwInput{KwData}{Algorithm 1}

\let\oldnl\nl
\newcommand{\nonl}{\renewcommand{\nl}{\let\nl\oldnl}}

\begin{document}

\begin{center}
  \begin{Large}
    \textbf{Compact MILP formulations for the $p$-center problem}
  \end{Large}
\vspace{.6cm}

  {Zacharie Ales$^1$, Sourour Elloumi$^1$}\vspace{.3cm}


    $^1$: ENSTA-ParisTech / UMA, 91762 Palaiseau, France\\
    Laboratoire CEDRIC, Paris, France\\
    \texttt{\{zacharie.ales, sourour.elloumi\}@ensta-paristech.fr}

\end{center}
\thispagestyle{empty}
\vspace{.3cm}

\vspace{.3cm}

\begin{center}
  \begin{minipage}{.7\linewidth}
    \textbf{Abstract.}  The $p$-center  problem consists  in selecting
    $p$ centers among $M$ to cover  $N$ clients, such that the maximal
    distance  between a  client  and its  closest  selected center  is
    minimized. For this problem we propose two new and compact integer
    formulations.

    Our   first  formulation   is   an  improvement   of  a   previous
    formulation. It significantly decreases  the number of constraints
    while preserving the optimal value  of the linear relaxation.  Our
    second formulation contains less  variables and constraints but it
    has a weaker linear relaxation bound.

    We  besides introduce  an algorithm  which enables  us to  compute
    strong   bounds  and   significantly  reduce   the  size   of  our
    formulations.

    Finally,  the  efficiency  of   the  algorithm  and  the  proposed
    formulations  are  compared in  terms  of  quality of  the  linear
    relaxation and computation time over instances from OR-Library.
  \end{minipage}
\end{center}

\section{Introduction}

We consider $N$ clients $\s{C_1, ..., C_N}$ and $M$ potential facility
sites $\s{F_1, ..., F_M}$. Let  $d_{ij}$ be the distance between $C_i$
and $F_j$.  The  objective of the $p$-center problem is  to open up to
$p$ facilities such that the maximal distance (called \textit{radius})
between a client and its closest selected site is minimized.

This problem  is very  popular in  combinatorial optimization  and has
many   applications.    We   refer   the   reader    to   the   recent
survey~\cite{Calik2015}.  Very recent publications
include~\cite{ferone2017b,ferone2017} which   provide   heuristic
solutions and~\cite{calik2013double} on an exact solution method.

In  this  paper, we  will  focus  on mixed-integer  linear  programming
formulations of the $p$-center problem. 

Let $\mathcal M$ and $\mathcal N$ respectively be the sets $\s{1, ..., M}$ and
$\s{1, ..., N}$. The most classical formulation, denoted by $(P_1)$, for the $p$-center problem (see for
example~\cite{daskin1995network}) considers the following variables:
\begin{itemize}
\item $y_j$ is a binary variable equal to $1$ if and only if $F_j$ is open;
\item $x_{ij}$ is a binary variable equal  to $1$ if and only if $C_i$
  is assigned to $F_j$;
\item $R$ is the radius.
\end{itemize}
\renewcommand{\arraystretch}{1.4}
\begin{center}
    \begin{subequations}
      \begin{empheq}[left=\fold\empheqlbrace]{align}
         \min &~R&\\
        \mbox{s.t.} & \displaystyle\sum_{j = 1}^M y_j\leq p & \label{eq:pc1}\\
        & \displaystyle\sum_{j = 1}^M x_{ij} = 1 & \label{eq:pc2} i\in \N  \\
        & x_{ij}\leq y_j & \label{eq:pc3}i\in \N, j\in\M \\
        & \displaystyle\sum_{j=1}^M d_{ij}~ x_{ij} \leq R & \label{eq:pc4}  i\in \N\\
        & x_{ij},y_j\in\s{0, 1} & i\in\N, j\in \M \notag\\
         & r\in\mathbb R\notag
      \end{empheq}
\label{eq:pc}
    \end{subequations}
       \end{center}

Constraint~\eqref{eq:pc1} ensures  that no  more than $p$  facilities are
opened.  Each client  is assigned  to exactly  one facility  through
Constraints~\eqref{eq:pc2}. Constraints~\eqref{eq:pc3} link variables
  $x_{ij}$ and $y_j$ while~\eqref{eq:pc4} ensure the coherence of the
  objective.

A more recent formulation, denoted by \fa, was proposed
in~\cite{elloumi2004new}. Let  $D^0 < D^1 < ... <
D^K$ be the different $d_{ij}$ values$~\forall i\in\mathcal N~\forall
j\in\mathcal M$. Note that, if many distances $d_{ij}$ have the
same   value,   $K$  may   be   significantly   lower  than   $M\times
N$. Let $\mathcal K$ be the set $\s{1, ..., K}$. Formulation~\fa~is based on the variables $y_j$, previously
introduced, and one binary variable $z^k$, for each $k\in\K$, equals to $1$ if
and only if the optimal radius is greater than or equal to $D^k$:

\begin{center}
  \begin{subequations}
    \begin{empheq}[left=(P_2)\empheqlbrace]{align}
       \min& ~D^0 + \sum_{k = 1}^K (D^k - D^{k-1})~z^k\\
      \mbox{s.t.} & ~1\leq\displaystyle\sum_{j=1}^M y_j \leq p& \label{eq:pcsc1}\\
       &   z^k  +   \sum_{j\,:\,d_{ij}<D^k}   y_j  \geq   1  &   
       i\in\N, k\in\K \label{eq:pcsc3}\\
      & y_j,z^k\in\s{0, 1} & j\in\M, k\in\K\notag
    \end{empheq}
  \end{subequations}
\end{center}

Constraints~\eqref{eq:pcsc3}  ensure that  if no  facility located  at
less than $D^k$ of client $C_i$ is selected, then the radius must be greater than or
equal to $D^k$.

This formulation has been proved to be tighter than $\fold$~\cite{elloumi2004new}. However, its size strongly depends on the value $K$
(\textit{i.e.,} the number of distinct distances $d_{ij}$).

It also has recently been  adapted to the $p$-dispersion problem which
consists in selecting  $p$ facilities among $N$ such  that the minimal
distance      between       two      selected       facilities      is
maximized~\cite{sayah2017new}.

A last formulation, that can be deduced  from $(P_2)$ by a change of variables, has been recently
introduced~\cite{calik2013double} and named $(P_4)$. It contains, for all $k\in\K$, a
binary variable $u_k$ equal to $1$ if and only if the optimal radius is
$D^k$ (i.e., $u_k= z^k-z^{k+1}$ and $z^k=\sum_{q=k}^K u_q$):

\begin{center}
  \begin{subequations}
    \begin{empheq}[left=(P_4)\empheqlbrace]{align}
       \min& \sum_{k = 1}^K D^ku_k\\
      \mbox{s.t.} & ~\eqref{eq:pcsc1}\notag\\
       & \sum_{j\,:\,d_{ij}\leq D^k} y_j \geq \sum_{q=1}^k u_q & 
       i\in\N, k\in\K \label{eq:pcru3}\\
       & \sum_{k=1}^K u_k = 1 & \label{eq:pcru4}\\
      & y_j,u_k\in\s{0, 1} & j\in\M, k\in\K\notag
    \end{empheq}
  \end{subequations}
\end{center}


They also  proposed a  weaker version  of this  formulation, called
$(P_3)$,    obtained   by    replacing    the    left-hand   side    of
constraints~\eqref{eq:pcru3} by $u_k$. They proved that
 $(P_4)$ leads to the same linear relaxation bound and has the same size as $(P_2)$.
\bigskip

The rest of the paper is organized as follows. Section~\ref{sec:ourformulations} presents our two new formulations.  In
Section~\ref{sec:iterative}     we     introduce    an    algorithm.   Finally,  Section~\ref{sec:results}   describes  numerical
results on instances from the OR-Library.

\section{Our new formulations}
\label{sec:ourformulations}
\subsection{Formulation \fb}

In \fa, for all $k\in\K$, variable $z^k$ is equal to $1$ if and only if the
optimal radius is greater than or equal to $D^k$.  As a consequence, the
following constraints are valid
\begin{equation}
  \label{eq:zordered}
  z^k\geq z^{k+1}\quad k\in\s{1, ..., K-1}.
\end{equation}

We first show that these inequalities are redundant for $(P_2)$. Let $(P_2')$ be the formulation obtained when
contraints~\eqref{eq:zordered}   are  added   to   $(P_2)$  and   let
$v(\overline{F})$ be the optimal value of
the linear  relaxation of a given  formulation $F$. We now  prove that
adding constraints~\eqref{eq:zordered}  does not improve the  quality of
the linear relaxation.

\begin{proposition}
$v(\overline{P_2'}) = v(\overline{P_2})$
\end{proposition}

\begin{proof}
  
We show that an optimal solution $(\tilde y,
\tilde z)$ of the relaxation of \fa~satisfies~\eqref{eq:zordered}.  For  each
distance $D^k$ there exists a client $i(k)$ such that 
\begin{equation}
\tilde z^{k}+ \sum_{j~:~d_{i(k)j}< D^{k}} \tilde y_j = 1
\label{eq:p1}
\end{equation}

otherwise $\tilde z^{k}$ can be decreased and $(\tilde y, \tilde z)$ is not
optimal.

We now assume that $\tilde z^{k-1} < \tilde z^{k}$ for some index $k\in\s{2, ...,
  K}$. It follows that
\begin{equation}
    {\tilde z^{k-1}}+  {\sum_{j~:~d_{i({k})j}< D^{k-1}} \tilde  y_j} <
    {\tilde z^{k}}+
    {\sum_{j~:~d_{i({k})j}< D^{k}} \tilde y_j}=1\nonumber
\end{equation}

The    last    equality   follows    from~\eqref{eq:p1}.    Therefore,
constraints~\eqref{eq:pcsc3} for $i(k)$ and $k-1$ is violated.
\end{proof}

We  now prove  that  a large  part  of constraints~\eqref{eq:pcsc3}  are
redundant in $(P_2')$.


Let $N_i^k$ be  the set of facilities located at  less than $D^k$ from
client $C_i$. We can observe  that $N_i^{k}$ is included in $N_i^{k+1}$,
for all $k\in\K$.  Moreover, $N_i^{k}$ is equal to  $N_i^{k+1}$ if and
only if there is no facility  at distance $D^{k}$ from client $C_i$. Let
$S_i$ be the set of indices $k\in\s{1, ..., K-1}$ such that $N_i^k$ is
different from $N_i^{k+1}$. Observe that $|S_i|\leq \min(M,K)$.
\bigskip

We define Formulation $(CP_1)$
as      Formulation     $(P_2')$      where      only     the
constraints~\eqref{eq:pcsc3} such that $k\in
S_i$ or $k=K$ are kept.

\begin{center}
  \begin{subequations}
    \begin{small}
      \begin{empheq}[left=(CP_1)\empheqlbrace]{align}
        \min &~D^0 + \sum_{k = 1}^K (D^k - D^{k-1})~z^k\\
        \mbox{s.t.}  &~\eqref{eq:pcsc1}, \eqref{eq:zordered}\notag& \\
&         z^k + \sum_{j\,:\,d_{ij}<D^k} y_j \geq 1 & i\in\N,~k\in
        S_{i}\cup\s{K} \label{eq:pcsco3}\\ 
& y_j,z^k\in\s{0, 1} &
        j\in\M, k\in\K\notag
      \end{empheq}
    \end{small}

  \end{subequations}
\end{center}

The   number  of   constraints   is  dominated   by   the  number   of
constraints~\eqref{eq:pcsco3}. This number is bounded by both $NM$ and $NK$. 

The following proposition proves that $(CP_1)$ is a valid formulation.

\begin{proposition}
  $(CP_1)$ is a valid formulation of the $p$-center problem.
\label{pr:pcscovalid}
\end{proposition}

\begin{proof}
  We   show  that   the   constraints  removed   from  $(P_2')$   are
  dominated.  If  $N_i^k=N_i^{k+1}$,  then  $\sum_{j\,:\,d_{ij}<D^{k}}
  y_j=\sum_{j\,:\,d_{ij}<D^{k+1}}  y_j$.   Since $z^{k}\geq  z^{k+1}$,
  we have:
\[   z^{k}   +   \sum_{j\,:\,d_{ij}<D^{k}}  y_j   \geq   z^{k+1}   +
\sum_{j\,:\,d_{ij}<D^{k+1}} y_j\geq 1.\]

As a consequence, the  constraint~\eqref{eq:pcsc3} associated with $i$
and $k$ is dominated by the one associated with $i$ and $k+1$.
\end{proof}


We now prove that Formulations $(P_2)$ and $(CP_1)$ lead to the same
bound by linear relaxation.

\begin{proposition}
$v(\overline{CP_1})=v(\overline{P_2})$.
\end{proposition}

\begin{proof}
The arguments  used in  the proof of  Proposition~\ref{pr:pcscovalid} can
 be used again to show that  the constraints removed from $(P_2')$ do
 not impact the value of the linear relaxation.


\end{proof}

To sum up, $(CP_1)$  is a valid formulation that has  the same LP bound
as $(P_2)$.  However, as detailed in Table~\ref{tab:size}, Formulation $(CP_1)$ is much smaller since it 
reduces the number of constraints by a factor of up to~$N$.

\subsection{Formulation \fc}

We now introduce a second  formulation, denoted by \fc, which contains
less variables and constraints than $(CP_1)$.

We replace the $K$
binary variable $z^k$ with a unique general integer variable $r$ which
represents the index of a radius:

\begin{center}
  \begin{subequations}
    \begin{empheq}[left=(CP_2)\empheqlbrace]{align}
       \min& ~r\notag\\
      \mbox{s.t.} & ~\eqref{eq:pcsc1}\notag& \\
       &   r  + k  \sum_{j\,:\,d_{ij}<D^k}   y_j  \geq  k  &   
       i\in\N, k\in S_i\cup\s{K}\label{eq:pcr3}\\
      & y_j\in\s{0, 1} & j\in\M\notag\\& r\in\s{0, ..., K}
\notag    \end{empheq}
  \end{subequations}
\end{center}

Constraints~\eqref{eq:pcr3} play a similar  role to
Constraints~\eqref{eq:pcsco3}.


Formulation $(CP_2)$ does not directly provide the value of the optimal
radius $R$ but its index $r$ such that $D^r=R$. We now prove that Formulation $(CP_2)$ is valid.

\begin{proposition}
 $(CP_2)$ is a valid formulation of the $p$-center problem.
\end{proposition}

\begin{proof}
  Let $(\tilde y, \tilde z)$ be an integer solution of $(CP_1)$. We first
  show that there exists an  integer solution $(\overline y, \overline r)$
  of $(CP_2)$  which provides the  same radius  by setting $\overline  y =
  \tilde y$ and $ \overline r = \sum_{k= 1}^K \tilde z^k$. We need to
  prove that constraints~\eqref{eq:pcr3} are satisfied. We know that 
  \begin{equation}
\tilde z^k + \sum_{j\,:\,d_{ij}<D^k} \tilde y_j \geq 1
\nonumber
\end{equation}

is satisfied for any client $C_{i}$ and any distance $D^{k}$.

If    $\tilde   z^k$    is   equal    to   $0$,    the   corresponding
Constraint~\eqref{eq:pcr3}  is  satisfied, as  $\sum_{j\,:\,d_{ij}<D^k}
\tilde y_j \geq 1$.  Otherwise, the  same result is obtained since the
$\tilde z^k$ variables are ordered in decreasing order which leads to $\overline
r\geq k$. These two solutions provide the same  radius as $D^0 + \sum_{k  = 1}^K (D^k -
D^{k-1})~\tilde z^k=D^{\sum_{k=1}^K\tilde z^k}$.

We now prove that for any solution $(\tilde y, \tilde r)$ of $(CP_2)$ there exists
an  equivalent solution  $(\overline y,  \overline  z)$ of  $(CP_1)$. We  set
$\overline y =\tilde y$ and $\overline z^k=1$ if and only if $\tilde r\geq
k$. Constraint
\begin{equation}
\tilde r + k \sum_{j\,:\,d_{ij}<D^k} \tilde y_j \geq k
\label{eq:proofpcrvalid}
\end{equation}

is satisfied for any $k\in\K$. If $\tilde r$ is lower than $k$, then at least
one variable $\tilde y_j$ from equation~\eqref{eq:proofpcrvalid} is
equal  to $1$  and the  corresponding constraint~\eqref{eq:pcsco3}  is
satisfied. Otherwise, $\overline z^k$ is equal to $1$ and the same conclusion is
reached.
\end{proof}

We now prove that the linear relaxation of $(CP_1)$ is stronger than  the one of $(CP_2)$.

\begin{assumption}
  We shall suppose $D^0=0$ and $\forall k\in\K,~D^k-D^{k-1}=1$.
\label{ass:rank}
\end{assumption}

This assumption is not restrictive,  one can transform any instance by
replacing any distance $D^k$ by  its rank $k$. The transformed problem
is equivalent as if the optimal radius is $D^{k^*}$, then
the optimal solution of the transformed problem is $k^*$.

Under  this assumption,  problems $(CP_1)$  and $(CP_2)$  have the  same
optimal values, both of them compute the rank of the optimal radius.

\begin{proposition}
  Let $\overline{CP_1}$ and $\overline{CP_2}$ respectively be the LP relaxation of $(CP_1)$ and $(CP_2)$, $v(\overline{CP_1})\geq v(\overline{CP_2})$ under Assumption~\ref{ass:rank}.
\end{proposition}

\begin{proof}
  Let $(\tilde y, \tilde z)$ be a solution of 
  $\overline{CP_1}$. We 
build a solution $(\overline y,  \overline r)$ of $\overline{CP_2}$ with
the same value. We take $\overline y =
  \tilde   y$    and   $\overline
  r=\sum_{k=    1}^K   \tilde   z^k$.

We need to
  prove that constraints~\eqref{eq:pcr3} are
  satisfied. 

Since  the $z^k$  variables  are  ordered in  decreasing order  by
Constraints~\ref{eq:zordered},  it   follows  that   $\overline  r\geq
k\tilde z^k$ $\forall k\in\K$.  This and
  Constraints~\eqref{eq:pcsc3}           imply          that
  Constraints~\eqref{eq:pcr3} are satisfied.
\end{proof}

Table~\ref{tab:size} summarizes the size of the previously mentioned formulations.

\begin{table}[h!]
  \centering
  \begin{tabular}{c@{\hspace{0.3cm}}c@{\hspace{0.3cm}}c}
    \hline \textbf{Formulation}& \textbf{\#  of variables}& \textbf{\#
                                                           of
                                                            constraints}\\\hline
$\fold$ & $\mathcal O(NM)$ & $\mathcal O(NM)$\\

\fa, $(P_3)$, $(P_4)$& $\mathcal O(M + K)$ & $\mathcal O(NK)$\\

\fb & $\mathcal O(M + K)$ & $\mathcal O(\min(NM,NK))$\\

\fc & $\mathcal O(M)$ & $\mathcal O(\min(NM,NK))$\\\hline
  \end{tabular}

  \caption{Size of the four formulations ($K\leq NM$).}
  \label{tab:size}
\end{table}


\section{A two-step resolution algorithm}
\label{sec:iterative}

We present, in this section, a two-step algorithm to solve more
efficiently the
$p$-center problem.

Let $lb$ be a lower bound of  the optimal radius. We suppose that $lb$
is one of the distances $D^k$ since, otherwise, $lb$ can be set to the next distance. All the distances $d_{ij}$ lower than
$lb$ can be replaced by $lb$. 

Similarly, all the distances $d_{ij}$ greater than an upper bound $ub$
can  be replaced  by $ub+1$  in order  to discard  solutions of  value
greater than $ub$. 

The size  of Formulations  \fa~and \fb~strongly  depends on  $K$. This
value  can be  reduced by  identifying lower  and upper  bounds.  Such
bounds can  easily be obtained,  as mentioned
in~\cite{elloumi2004new}.  \bigskip

Our resolution  algorithm, depicted  in Figure~\ref{fig:algo},  can be
applied to  any formulation  $F$ of  the $p$-center  problem including
$(P_1)$, $(P_2)$, $(P_3)$, $(P_4)$, $(CP_1)$ and $(CP_2)$. It is mainly
based
on the idea that whenever the optimal value $\overline v$ of the linear
relaxation of~$F$ is
not equal to an existing
distance, then there exists $k\in K$ such that $D^{k-1}<\overline v <
D^k$.  In that  case, $D^k$  constitutes a  stronger lower  bound than
$\overline v$ and the linear 
relaxation can be solved again.  This process is repeated until an existing
 distance   is  obtained   as  the   optimal  value   of  the   linear
 relaxation. This constitutes Step~1 of the algorithm.

The bound  obtained when applying this algorithm
over~\fa~or \fb~corresponds  to the  one called  $LB^*$, computed  by a
binary search algorithm in~\cite{elloumi2004new}.

Step 1 can be further improved by introducing the notion
of \textit{dominated clients} and \textit{dominated facilities} within
some reduction rules.  A facility $F_a$ is dominated
if  there  exists another  facility  $F_b$ such  that  $d_{ia}\geq
d_{ib}$ for all clients $i$. Such a facility can be removed as it
will always be at least as interesting to assign a client to
 $F_b$ than to $F_a$.  Similarly, a
client  $C_a$ is  said to  be dominated  if there  exists another
client   $C_b$  such   that  $d_{aj}\leq   d_{bj}$  for   all  facilities
$j$. Dominated clients can also be ignored.

Instructions~3 and~4 are repeated since new dominated
clients and facilities may be found when a bound is improved,
and vice versa.  

Step  2  of  Algorithm  1   consists  in  solving  Formulation~$F$  to
optimality with the improved bounds $lb$ and $ub$ computed in Step~1.

\begin{figure}[h!]
  \centering
  \begin{algorithm}[H]
  \SetAlgoVlined
  \DontPrintSemicolon

  \KwData{\\$F$: formulation of the $p$-center problem\\$p$: maximal number of centers\\ $d$: distances\\$lb$,
    $ub$: initial bounds} 
  \KwResult{The optimal radius\\} 

\nonl// Step 1\;
  \Repeat{$\overline v=lb$ // until $\overline v$ is one of the existing distances}{

    \Repeat{$lb$ and $ub$ are not improved and no more dominated
    clients or facilities have been found}
    {
      Remove dominated clients and facilities // Reduction rules\;
      $(lb, ub)\leftarrow$ Compute bounds\;
    }
    $\overline v\leftarrow$ SolveLinearRelaxation($F$, $lb$, $ub$)\;


     $lb\leftarrow \min_k \s{D^k~:~\overline v\leq D^k}$\;

  }
\nonl // Step 2\;
  $r^*\leftarrow $ SolveOptimally($F$, $lb$, $ub$)\;
 \Return{$r^*$}\;
\end{algorithm}

  \caption{Algorithm  used to  solve the  $p$-center problem
    through F, a $p$-center formulation.}
  \label{fig:algo}
\end{figure}

\section{Numerical results}
\label{sec:results}
We implement Formulations $(P_1)$, $(P_2)$, $(CP_1)$
and $(CP_2)$ as well as Algorithm 1 on an Intel XEON E3-1280 with 3,5 GHz and
32Go of RAM with the Java API of CPLEX 12.7. Following
several authors, we consider instances from the OR-Library~\cite{beasley1990or}.

\subsection{Comparing sizes and computation times on $5$ instances}

Table~\ref{tab:size2} presents a  comparison of the sizes  of the four
formulations  on  the   five  first  instances  of   the  OR-Library  with
$N=M=100$.  We use the initial lower bound
$LB_0=\max_{i\in\N}\min_{j\in\M}  d_{ij}$  and   initial  upper  bound
$UB_0=\min_{j\in\M}\max_{i\in\N}           d_{ij}$          introduced
in~\cite{elloumi2004new}.

As expected, the number of variables in $(CP_1)$ and $(P_2)$ are equal
and are significantly  lower than in $(P_1)$.  Formulation $(P_2)$ has
more constraints than Formulation $(P_1)$.  Formulation $(CP_1)$ has
by far less constraints than  $(P_2)$. All
this explains why $(CP_1)$ has the best performances in every aspect.

Formulation $(CP_2)$ is the most compact but this does not fully compensate
the poor quality of its LP bound.

\begin{table}\centering
  \begin{small}
    \begin{tabular}{l@{\hspace{0.1cm}}|@{\hspace{0.2cm}}lR{1.3cm}*{6}{R{1.5cm}}}
      \hline
      &    \textbf{}   &    {$\mathbf{(P_1)}$}   &    {$\mathbf{(P_2)}$}   &
                                                                             {$\mathbf{(CP_1)}$} & {$\mathbf{(CP_2)}$}\\\hline

{\textbf{Instance 1}}       & {number of variables} & 10101 & 286 & 286 & {101}\\
& {number of constraints} &12209 & 18602 & 6089 &{5903}\\
{($LB_0 = 0$)}  & LP bound & 97,57 & 106,54 & 106,54 & 83,62\\
{($UB_0 = 186$)}& {resolution time (s)} &9,14 & 251,28 & \textbf{3,16} & 14,94\\\hline

{\textbf{Instance 2}} & {number of variables} & 10101 & 277 & 277 & {101}\\
     & {number of constraints} &    12473 & 17702 & 6094 & {5917}\\
($LB_0 = 0$) &LP bound & 76,72 & 85,68 & 85,68 & 70,19\\
($UB_0 = 178$)    & {resolution time (s)} & 15,69 & 47,31 & \textbf{2,99} & 19,80\\\hline

{\textbf{Instance 3}} & {number of variables} & 10101 & 305 & 305 & {101}\\
     & {number of constraints} &	11293 & 20502 & 6852 & {6647}\\
($LB_0 = 0$) & LP bound & 73,24 & 83,28 & 83,28 & 68,92\\
($UB_0 = 205$)    & {resolution time (s)} & 11,68 & 21,02 & \textbf{2,85} & 10,99\\\hline

{\textbf{Instance 4}} & {number of variables} & 10101 & 299 & 299 & {101}\\
     & {number of constraints} &	12009 & 19902 & 6403 & {6204}\\
($LB_0 = 0$) & LP bound & 54,55 & 64,16 & 64,16 & 52,42\\
($UB_0 = 204$)   & {resolution time (s)} & 3,19 & 43,02 & \textbf{1,64} & 12,90\\\hline

{\textbf{Instance 5}} & {number of variables} & 10101 & 270 & 270 & {101}\\
     &{number of constraints} &	11777 & 17002 & 6263 & {6093}\\
($LB_0 = 0$) & LP bound & 30,37 & 37,82 & 37,82 & 29,29\\
($UB_0 = 169$)   & {resolution time (s)} & 1,93 & 25,10 & \textbf{1,66} & 11,65\\\hline
    \end{tabular}
  
    \caption{Size and resolution times (1 thread) of the formulations for the
      five   first    OR-Library   instances   with    $lb=LB_0$   and
      $ub=UB_0$.}
    \label{tab:size2}
  \end{small}

\end{table}

\subsection{Relaxation and computation times on the $40$ OR-Library instances}

In Table~\ref{tab:res1}, we perform a larger comparison with  stronger
bounds $lb$ and $ub$ equal to the bounds  $LB_1$ and $UB_1$ introduced in~\cite{elloumi2004new}. The resolution is then
performed by CPLEX with its default  parameters but with a maximal CPU
time of $1$ hour.

The first  column is  the instance number.  The three  following columns
provide $N$,  $p$ and  the optimal  value of  the instances  ($N=M$ in
these instances). Columns 5 and 6 contain the initial bounds $LB$ and
$UB$. For each formulation, column ``b'' corresponds the optimal value of
    the linear relaxation  and column ``t'' to the  resolution time in
    seconds. 

We can first observe that  Formulations $(CP_1)$ and $(P_2)$ solve all
the $40$ instances within $1$ hour  while ten instances are not solved
 with $(P_1)$ and one  instance is not solved with $(CP_2)$. We can
 even observe that $(CP_1)$ solves the whole set of  instances in less than $50$
 minutes and $(P_2)$ in less than $85$ minutes. 

Formulation $(P_2)$ outperforms $(CP_1)$  mainly on instances $36$ and
$39$. This is possibly due to some difficulty of the solver to
find good feasible solutions.


\renewcommand{\arraystretch}{1.2}

\begin{table}
  \centering
  \begin{tabular}{r@{\hspace{0.1cm}}|@{\hspace{0.1cm}}rr@{\hspace{0.1cm}}|@{\hspace{0.1cm}}r@{\hspace{0.1cm}}|*{2}{r}@{\hspace{0.1cm}}*{4}{|R{0.6cm}R{1cm}@{\hspace{0.1cm}}}}
    &  \multirow{2}{*}{\textbf{N}}   &   \multirow{2}{*}{\textbf{p}}  &   \multirow{2}{*}{\textbf{opt}}   &   \multirow{2}{*}{\textbf{$\mathbf{lb}$}}  &
                                                              \multirow{2}{*}{\textbf{$\mathbf{ub}$}}
    & \multicolumn{2}{M{2cm}|}{$\mathbf{(P_1)}$} & \multicolumn{2}{M{2cm}|}{$\mathbf{(P_2)}$}  & \multicolumn{2}{M{2cm}|}{$\mathbf{(CP_1)}$} & \multicolumn{2}{M{2cm}}{$\mathbf{(CP_2)}$}\\
    & & & & & & \textbf{b} & \textbf{t} & \textbf{b} & \textbf{t} & \textbf{b} & \textbf{t}  & \textbf{b} & \textbf{t}\\\hline

1 & 100 & 5 & 127 & 59 & 133 & 98 & 2,4 & 107 & 75,3 & 107 & \textbf{1,0} & 85 & 4,0\\
2 & 100 & 10 & 98 & 56 & 117 & 77 & 2,9 & 86 & 7,3 & 86 & \textbf{0,5} & 71 & 5,2\\
3 & 100 & 10 & 93 & 55 & 116 & 74 & 2,9 & 84 & 2,5 & 84 & \textbf{0,2} & 69 & 3,1\\
4 & 100 & 20 & 74 & 41 & 127 & 55 & 0,7 & 65 & 7,9 & 65 & \textbf{0,6} & 53 & 3,4\\
5 & 100 & 33 & 48 & 23 & 87 & 31 & 0,8 & 38 & 1,0 & 38 & \textbf{0,1} & 30 & 1,5\\\hline
6 & 200 & 5 & 84 & 38 & 94 & 68 & 35,9 & 75 & 106,7 & 75 & \textbf{2,7} & 59 & 47,1\\
7 & 200 & 10 & 64 & 34 & 79 & 51 & 20,5 & 58 & 100,2 & 58 & \textbf{1,8} & 46 & 26,1\\
8 & 200 & 20 & 55 & 30 & 72 & 41 & 20,7 & 48 & 87,2 & 48 & \textbf{1,6} & 38 & 19,6\\
9 & 200 & 40 & 37 & 22 & 73 & 28 & 8,9 & 33 & 14,9 & 33 & \textbf{1,4} & 27 & 29,8\\
10 & 200 & 67 & 20 & 11 & 44 & 15 & 1,6 & 18 & 0,8 & 18 & \textbf{0,3} & 14 & 5,5\\\hline
11 & 300 & 5 & 59 & 34 & 67 & 50 & 99,0 & 54 & 30,4 & 54 & \textbf{6,2} & 44 & 68,1\\
12 & 300 & 10 & 51 & 30 & 72 & 43 & 229,7 & 48 & 71,0 & 48 & \textbf{7,2} & 39 & 98,7\\
13 & 300 & 30 & 36 & 20 & 56 & 28 & 114,0 & 33 & 44,6 & 33 & \textbf{4,7} & 26 & 106,9\\
14 & 300 & 60 & 26 & 14 & 60 & 19 & 157,1 & 23 & 33,4 & 23 & \textbf{12,9} & 18 & 151,7\\
15 & 300 & 100 & 18 & 10 & 42 & 13 & 8,6 & 16 & 9,4 & 16 & \textbf{0,9} & 13 & 30,2\\\hline
16 & 400 & 5 & 47 & 26 & 51 & 41 & 403,2 & 45 & 25,3 & 45 & \textbf{3,3} & 36 & 54,5\\
17 & 400 & 10 & 39 & 21 & 47 & 33 & 737,8 & 36 & 35,0 & 36 & \textbf{24,9} & 29 & 149,2\\
18 & 400 & 40 & 28 & 16 & 50 & 22 & 664,7 & 25 & 96,4 & 25 & \textbf{22,1} & 20 & 431,4\\
19 & 400 & 80 & 18 & 10 & 40 & 14 & 226,2 & 16 & 81,4 & 16 & \textbf{18,5} & 13 & 116,9\\
20 & 400 & 133 & 13 & 7 & 32 & 10 & 9,0 & 12 & 3,0 & 12 & \textbf{0,9} & 10 & 22,5\\\hline
21 & 500 & 5 & 40 & 23 & 48 & 35 & 2581,0 & 37 & 118,3 & 37 & \textbf{13,6} & 31 & 194,6\\
22 & 500 & 10 & 38 & 21 & 49 & 31 & - & 35 & 924,4 & 35 & \textbf{24,6} & 28 & 507,8\\
23 & 500 & 50 & 22 & 13 & 38 & 17 & 1375,8 & 20 & 212,2 & 20 & \textbf{38,4} & 16 & 481,8\\
24 & 500 & 100 & 15 & 9 & 35 & 12 & 573,7 & 14 & 51,0 & 14 & \textbf{29,6} & 11 & 209,2\\
25 & 500 & 167 & 11 & 6 & 27 & 8 & 57,2 & 10 & 5,1 & 10 & \textbf{2,0} & 8 & 23,1\\\hline
26 & 600 & 5 & 38 & 21 & 43 & 32 & 3093,6 & 35 & 106,0 & 35 & \textbf{13,6} & 28 & 152,4\\
27 & 600 & 10 & 32 & 18 & 39 & 28 & 3118,9 & 30 & 104,3 & 30 & \textbf{48,3} & 25 & 341,5\\
28 & 600 & 60 & 18 & 10 & 33 & 14 & - & 16 & 176,2 & 16 & \textbf{103,3} & 13 & -\\
29 & 600 & 120 & 13 & 7 & 36 & 10 & - & 12 & 130,7 & 12 & \textbf{77,8} & 9 & 893,6\\
30 & 600 & 200 & 9 & 5 & 29 & 7 & 106,5 & 8 & \textbf{12,4} & 8 & 15,7 & 7 & 89,8\\\hline
31 & 700 & 5 & 30 & 16 & 34 & 27 & 1793,8 & 28 & 68,8 & 28 & \textbf{12,5} & 24 & 139,9\\
32 & 700 & 10 & 29 & 16 & 35 & 25 & - & 27 & 718,7 & 27 & \textbf{127,3} & 22 & 944,5\\
33 & 700 & 70 & 15 & 9 & 26 & 13 & - & 14 & 155,1 & 14 & \textbf{76,0} & 12 & 890,1\\
34 & 700 & 140 & 11 & 6 & 30 & 9 & 2617,9 & 10 & 168,7 & 10 & \textbf{32,8} & 8 & 464,9\\
35 & 800 & 5 & 30 & 16 & 32 & 27 & - & 29 & 23,0 & 29 & \textbf{13,0} & 23 & 170,6\\\hline
36 & 800 & 10 & 27 & 16 & 34 & 24 & - & 26 & \textbf{130,3} & 26 & 821,7 & 21 & 1056,6\\
37 & 800 & 80 & 15 & 8 & 26 & 12 & - & 14 & 222,5 & 14 & \textbf{90,9} & 11 & 1706,9\\
38 & 900 & 5 & 29 & 15 & 35 & 25 & - & 27 & 68,8 & 27 & \textbf{19,0} & 21 & 300,1\\
39 & 900 & 10 & 23 & 13 & 28 & 20 & - & 22 & \textbf{348,4} & 22 & 1190,0 & 18 & 1786,4\\
40  & 900  &  90 &  13  & 7  &  22 &  10  & -  &  12 &  551,0  & 12  &
                                                         \textbf{129,5} & 10 & 1059,9\\\hline
&& & & & \textbf{Total} & & 57699 & & 5129 & & \textbf{2991} & & 16390
  \end{tabular}
  \caption{Comparison of  the different formulations with  $lb = LB_1$
    and $ub=UB_1$.  For each
    instance, the smallest time appears in bold. Symbol ``-'' means that the instance was not solved within
    $1$ hour.}
  \label{tab:res1}
\end{table}

\subsection{Results of Algorithm 1}

Table~\ref{tab:res2} presents the results of Algorithm 1
with formulations $(CP_1)$ and $(CP_2)$.  
    Columns  ``t1''   and   ``t2''
    respectively correspond to the
    time of the first phase and the total time.

Formulation $(CP_2)$ is now able to  solve all the instances within $1$
hour. We observe that the total time to solve the $40$ instances is reduced
by approximately $6$ times for $(CP_1)$ and $14$ times for $(CP_2)$ if
compared to Table~\ref{tab:res1}. 

\begin{table}
  \centering
  \begin{tabular}{*{3}{r}@{\hspace{0.1cm}}|@{\hspace{0.1cm}}
r@{\hspace{0.1cm}}|@{}H@{}H@{}
@{}H@{}H@{}
@{\hspace{0.1cm}}R{1cm}R{1cm}@{\hspace{0.1cm}}|
R{1cm}R{1cm}
}
    &  \multirow{2}{*}{\textbf{N}}   &   \multirow{2}{*}{\textbf{p}}  &   \multirow{2}{*}{\textbf{opt}}   &   \multirow{2}{*}{\textbf{lb}}  &
\multirow{2}{*}{\textbf{ub}}
 &
\multirow{2}{*}{\textbf{LB2}}  &
\multirow{2}{*}{\textbf{UB2}}
  & \multicolumn{2}{M{2cm}|}{$\quad\mathbf{(CP_1)}$} & \multicolumn{2}{M{2cm}}{$\quad\mathbf{(CP_2)}$}\\
    & & & & & & & & \textbf{t1} & \textbf{t2} & \textbf{t1} & \textbf{t2} \\\hline
1 & 100 & 5 & 127 & 59 & 133 & 121 & 133 & 0,2 & \textbf{0,3} & 0,3 & 0,7\\
2 & 100 & 10 & 98 & 56 & 117 & 98 & 117 & 0,2 & \textbf{0,2} & 0,3 & 0,4\\
3 & 100 & 10 & 93 & 55 & 116 & 93 & 116 & 0,2 & \textbf{0,3} & 0,3 & 0,4\\
4 & 100 & 20 & 74 & 41 & 127 & 74 & 127 & 0,3 & \textbf{0,4} & 0,4 & 0,5\\
5 & 100 & 33 & 48 & 23 & 87 & 48 & 87 & 0,1 & \textbf{0,2} & 0,3 & 0,4\\\hline
6 & 200 & 5 & 84 & 38 & 94 & 83 & 94 & 1,9 & \textbf{2,7} & 5,2 & 6,3\\
7 & 200 & 10 & 64 & 34 & 79 & 64 & 79 & 1,1 & \textbf{1,4} & 3,0 & 3,4\\
8 & 200 & 20 & 55 & 30 & 72 & 55 & 72 & 0,8 & \textbf{1,0} & 2,8 & 3,0\\
9 & 200 & 40 & 37 & 22 & 73 & 37 & 73 & 2,0 & \textbf{2,7} & 4,5 & 5,4\\
10 & 200 & 67 & 20 & 11 & 44 & 20 & 44 & 0,4 & \textbf{0,6} & 0,9 & 1,1\\\hline
11 & 300 & 5 & 59 & 34 & 67 & 59 & 61 & 0,8 & \textbf{0,9} & 2,2 & 2,2\\
12 & 300 & 10 & 51 & 30 & 72 & 51 & 72 & 3,4 & \textbf{4,6} & 10,2 & 12,5\\
13 & 300 & 30 & 36 & 20 & 56 & 36 & 56 & 3,6 & \textbf{4,6} & 8,8 & 9,8\\
14 & 300 & 60 & 26 & 14 & 60 & 26 & 60 & 3,5 & \textbf{4,5} & 14,8 & 17,5\\
15 & 300 & 100 & 18 & 10 & 42 & 18 & 42 & 1,5 & \textbf{2,1} & 3,3 & 3,7\\\hline
16 & 400 & 5 & 47 & 26 & 51 & 47 & 51 & 1,4 & \textbf{1,4} & 6,4 & 6,4\\
17 & 400 & 10 & 39 & 21 & 47 & 39 & 47 & 3,3 & \textbf{4,3} & 9,5 & 10,6\\
18 & 400 & 40 & 28 & 16 & 50 & 28 & 50 & 5,8 & \textbf{8,3} & 29,1 & 33,3\\
19 & 400 & 80 & 18 & 10 & 40 & 18 & 40 & 4,1 & \textbf{6,2} & 9,8 & 12,1\\
20 & 400 & 133 & 13 & 7 & 32 & 13 & 32 & 2,5 & \textbf{3,0} & 4,0 & 5,0\\\hline
21 & 500 & 5 & 40 & 23 & 48 & 40 & 48 & 3,1 & \textbf{4,0} & 9,7 & 10,3\\
22 & 500 & 10 & 38 & 21 & 49 & 38 & 49 & 16,6 & \textbf{26,5} & 38,6 & 48,3\\
23 & 500 & 50 & 22 & 13 & 38 & 22 & 38 & 7,0 & \textbf{9,9} & 31,5 & 37,1\\
24 & 500 & 100 & 15 & 9 & 35 & 15 & 35 & 7,6 & \textbf{11,4} & 18,5 & 23,7\\
25 & 500 & 167 & 11 & 6 & 27 & 11 & 27 & 3,7 & \textbf{4,6} & 7,5 & 9,0\\\hline
26 & 600 & 5 & 38 & 21 & 43 & 37 & 43 & 4,6 & \textbf{5,3} & 19,3 & 20,7\\
27 & 600 & 10 & 32 & 18 & 39 & 32 & 39 & 9,5 & \textbf{12,5} & 23,0 & 26,2\\
28 & 600 & 60 & 18 & 10 & 33 & 18 & 33 & 14,4 & \textbf{17,5} & 42,0 & 48,7\\
29 & 600 & 120 & 13 & 7 & 36 & 13 & 36 & 23,4 & \textbf{32,7} & 91,0 & 111,4\\
30 & 600 & 200 & 9 & 5 & 29 & 9 & 29 & 10,5 & \textbf{15,1} & 17,4 & 21,9\\\hline
31 & 700 & 5 & 30 & 16 & 34 & 30 & 34 & 8,2 & \textbf{9,3} & 15,8 & 17,5\\
32 & 700 & 10 & 29 & 16 & 35 & 28 & 35 & 18,8 & \textbf{71,8} & 33,8 & 109,8\\
33 & 700 & 70 & 15 & 9 & 26 & 15 & 26 & 10,2 & \textbf{14,3} & 25,4 & 34,4\\
34 & 700 & 140 & 11 & 6 & 30 & 11 & 30 & 34,2 & \textbf{46,4} & 90,1 & 107,6\\
35 & 800 & 5 & 30 & 16 & 32 & 30 & 32 & 2,2 & \textbf{2,2} & 11,8 & 12,0\\\hline
36 & 800 & 10 & 27 & 16 & 34 & 27 & 34 & 20,0 & \textbf{30,3} & 40,5 & 53,1\\
37 & 800 & 80 & 15 & 8 & 26 & 15 & 26 & 21,8 & \textbf{27,8} & 50,2 & 60,9\\
38 & 900 & 5 & 29 & 15 & 35 & 29 & 35 & 12,2 & \textbf{12,7} & 29,7 & 30,3\\
39 & 900 & 10 & 23 & 13 & 28 & 23 & 28 & 36,6 & \textbf{49,7} & 45,5 & 153,4\\
40 & 900 & 90 & 13 & 7 & 22 & 13 & 22 & 21,8 & \textbf{31,2} & 50,3 & 70,7\\\hline
\multicolumn{2}{l}{\textbf{Total}}&& & & & & & & \textbf{484} & & 1142
  \end{tabular}
  \caption{Results obtained with Algorithm 1 of
    Figure~\ref{fig:algo} with $lb = LB_1$
    and $ub=UB_1$.}
  \label{tab:res2}
\end{table}

\section{Conclusion}

We introduced two new compact formulations of the $p$-center problem. We
theoretically compared the quality of their LP bounds and their sizes to existing
formulations. Numerical experiments confirmed these results and highlighted
the fact that  our new formulation $(CP_1)$  outperforms the previously
known formulations  $(P_1)$ and $(P_2)$ at all levels. Our more compact formulation $(CP_2)$
suffers from the poor quality of its linear relaxation. Another aspect
of our work was to embed the formulations within a two-step
algorithm in order to obtain better computation times.

Our  future work  will focus  on improving  our
compact formulation through polyhedral studies.


\bibliographystyle{plain}
\bibliography{bibliographie}

\end{document}